\documentclass{amsart}
\usepackage{graphicx,xcolor} 
\usepackage{amssymb}
\usepackage{amsmath}
\usepackage{stmaryrd}
\usepackage{blindtext}
\usepackage{amsthm}
\usepackage[hidelinks]{hyperref}
\usepackage[english]{babel}
\usepackage{enumitem}
\usepackage[all]{xy}
\usepackage{mathrsfs}
\usepackage{geometry}
\DeclareMathOperator{\C}{\mathbb C}

\DeclareMathOperator{\Z}{\mathbb Z}
\DeclareMathOperator{\N}{\mathbb N}

\DeclareMathOperator{\g}{\mathfrak g}

\DeclareMathOperator{\ssl}{\mathfrak{sl}}

\DeclareMathOperator{\A}{\mathscr A}

\newcommand{\co}{\left[}
\newcommand{\cf}{\right]}
\newcommand{\po}{\left(}
\newcommand{\pf}{\right)}

\newcommand{\keyword}[1]
{
  \small	
  \textbf{\textit{Keywords---}} #1
}

\newtheorem{theoreme}{Theorem}[]
\newtheorem*{maintheorems}{Main theorem}
\newtheorem{corollaire}{Corollary}[theoreme]

\newtheorem{proposition}[theoreme]{Proposition}
\newtheorem{definition}[theoreme]{Definition}

\newtheorem{remark}[]{Remark}

\title{Compatibility between truncation and coproducts for quantum affine algebra and Yangian of $\ssl_2(\C)$}
\author{Jérôme Milot \\ \tiny Université de Lille, France}
\address{Laboratoire Paul Painlevé, Université de Lille, France}
\email{jerome.milot@univ-lille.fr}

\begin{document}

\begin{abstract}
    We prove that the standard Drinfeld-Jimbo coproducts for Yangians and quantum affine algebras factorize through their truncated quotients in the case of $\ssl_2(\C)$. As an auxiliary result, we give formulas for the coproduct of the truncation series in both cases.
\end{abstract}

\maketitle
\keyword{Drinfeld-Jimbo quantum groups, Truncated Yangians, Truncated quantum affine algebras}

\section*{Introduction}
Let $\g$ be a complex finite-dimensional simple Lie algebra and $q$ a complex number which is not a root of unity. Let $Y(\g)$ be the Yangian and $U_q(\hat\g)$ the quantum affine algebra constructed by Drinfeld \cite{NewDrinfeld} as Hopf algebra deformations of the current algebra $\g[z]$ and the loop algebra $\g[z,z^{-1}]$ respectively.  

We are interested in certain quotients of Yangians and quantum affine algebras, called \textit{truncated Yangians} and \textit{truncated quantum affine algebras} respectively . Both quotients depend on a complex polynomial $a(z)$ and are denoted $Y^a(\g)$ and $U_q^a(\hat{\g})$. Their relations are given by some elements in the commutative Drinfeld-Cartan subalgebra of $Y(\g)$ and $U_q(\hat \g)$. Truncated Yangians in type A were introduced first by Brundan-Kleshchev \cite{Brunchev} to study representation theory of finite W-algebras, then generalized to other types in the context of quantized affine Grassmanian slices by Kamnitzer-Webster-Weekes-Yacobi \cite{KWWY} and in the study of quantized Coulomb branches of 3d N = 4 SUSY quiver gauge theories by Braverman-Finkelberg-Nakajima \cite{BFN}. Truncated quantum affine algebras were defined by Finkelberg-Tsymbaliuk \cite{Finkeliuk} to realize quantized K-theoretical Coulomb branches. 

We are interested in the following natural question : are the defining relations of truncated Yangians and truncated quantum affine algebras preserved by the standard Drinfeld-Jimbo coproduct of the two quantum groups ? In this paper, we confirm the first non trivial case $\g = \ssl_2$. Our main result is stated as follows (Theorem \ref{compatibilityangian} for the Yangian and Theorem \ref{compatibilityaffine} for the quantum affine algebra) :
\begin{maintheorems}
    Let $a(z)$ and $b(z)$ be two complex polynomials.
    \begin{itemize}
        \item The Drinfeld-Jimbo coproduct of the Yangian $Y(\ssl_2)$ induces an algebra morphism :
    \[ Y^{ab}(\ssl_2) \longrightarrow Y^a (\ssl_2)\otimes Y^b(\ssl_2).\]
        \item Similarly, the Drinfeld-Jimbo coproduct of the quantum affine algebra $U_q(\hat{\ssl_2})$ induces an algebra morphism :
    \[ U_q^{ab}(\hat{\ssl_2}) \longrightarrow U_q^a(\hat{\ssl_2}) \otimes U_q^b(\hat{\ssl_2}).\]
    \end{itemize}
\end{maintheorems}
 In recent works of \cite{KTWWY1,KTWWY2,Hernandez, ZhangHernandez} many interesting representations of Yangians and quantum affine algebras are shown to factorize through truncations. Our result enables to prove that all tensor products of representations therein factorize through truncations for $\ssl_2$.

Sections 1-3 are about the Yangian case. In section 1, we briefly recall the definition of truncated Yangian of $\ssl_2$. In section 2, we give a polynomial formula for the coproduct of truncation series in this case. In section 3, we use this formula to establish the compatibility between truncation and coproduct for the Yangian. 

Sections 4-6 deal with the quantum affine algebra. Here, the main difference in the computation in the proofs is that we use of q-exponentials instead of exponentials. 

Section 7 is a direct generalization of these results to the shifted Yangians and shifted quantum affine algebras. We omit their proofs because they follow from the non-shifted cases by standard zigzag arguments.

\subsection*{Acknowledgements}
The author acknowledges the support of CDP C2EMPI, as well as the French State under the France-2030 program, the University of Lille, the Initiative of Excellence of the University of Lille, the European Metropolis of Lille for their funding and support of the R-CDP-24-004-C2EMPI project.

\section{Yangian}
Fix $\ssl_2$ the Lie algebra defined by generators  $x^+,x^-,h$ and relations : 
\[ \co x^+,x^- \cf = h, \quad \co h,x^\pm \cf = \pm 2x^\pm.\]

\begin{definition}
    The \textbf{Yangian} of $\ssl_2$ is the associative algebra $Y$, defined by generators $x^+_n,x^-_n,h_n$ for $n \in \N$ subject to the following relations for $r,s \in \N$:
    \begin{gather*}
        \co h_r,h_s \cf = 0, \quad \co h_0,x^\pm_s \cf = \pm 2x^\pm_s, \\
        \co h_{r+1},x^\pm_s \cf - \co h_r,x^\pm_{s+1}   \cf = \pm\po h_rx^\pm_s + x^\pm_sh_r \pf, \\
        \co x^+_r,x^-_s \cf = h_{r+s}, \\
        \co x^\pm_{r+1},x^\pm_s \cf - \co x^\pm_r,x^\pm_{s+1}   \cf = \pm\po x^\pm_rx^\pm_s + x^\pm_sx^\pm_r \pf.
    \end{gather*}
\end{definition}

Define the generating series :
\[ x^\pm(z) := \sum\limits_{n\in\N} x^\pm_nz^{-n-1}, \quad h(z) := 1 + \sum\limits_{n\in\N} h_nz^{-n-1}.\]

Denote by $Y^=$ the subalgebra of $Y$ generated by $h_n$ for $n \in \N$.

The Yangian has a Hopf algebra structure whose coproduct is determined by :
    \begin{align*} 
    \Delta(x^\pm_{0}) &= x^\pm_0\otimes 1 + 1\otimes x_0^\pm; \quad \Delta(h_0) = h_0\otimes 1 + 1 \otimes h_0 \\
    \Delta(x^+_1) &= x_1^+ \otimes 1 + 1 \otimes x_1^+ + h_0 \otimes x_0^+. 
    \end{align*}

\begin{remark}
    There exists a general formula for the coproduct of $h(z)$ \cite[Definition 2.24]{Molev} :
    \[ \Delta(h(z)) = \sum\limits_{k=0}^\infty(-1)^k(k+1)(x^-(z+1))^kh(z) \otimes h(z) (x^+(z+1))^k .\] 
    This is one of our motivations for the study of modified series, which should have a simpler coproduct. 
\end{remark}

Let us define the GKLO series in order to compute their coproduct. 

\begin{definition}\cite{GKLO,KWWY,BFN}
    Let $m \in \N$ and $a(z)$ be  complex Laurent series in $z^{-1}$ of the leading term $z^{2m}$. The Gerasimov–Kharchev–Lebedev–Oblezin series, also called $A$-series, associated to $a(z)$ is the Laurent series in $z^{-1}$ of the leading term $z^m$, with coefficients in $Y^=$, uniquely determined by :
    \[ h(z) = \frac{a(z)}{A^a(z)A^a(z-1)} \in 1 + z^{-1}Y^=[[z^{-1}]].\]
    One can write :
    \[ A^a(z) = z^m + \sum\limits_{k \leqslant m} A^a_{-k}z^{k-1}.\]
    The \textbf{truncated Yangian} $Y^a$ is the quotient of $Y$ by the two-sided ideal generated by $(A^a_k)_{k\geqslant 0}$.
\end{definition}

The GKLO series were initially introduced in \cite[Lemma 2.1]{GKLO} and the definition of truncated Yangians was given in \cite[Theorem 4.5]{KWWY} and \cite[Theorem B.15]{BFN}.

\begin{remark}
    The algebra $Y^a$ is completely characterized by the fact that the image of the series $A(z)$ is a polynomial in $z$. We will take advantage of this remark to check the compatibility between the coproduct and the truncation.
\end{remark}

\section{Coproduct of A-series}
We start by computing a coproduct formula for $A^1(z)$, using purely algebraic arguments. In fact, from this formula we will deduce the general formula for $A^a(z)$ for an arbitrary $a(z)$. Let then $A(z) := A^1(z)$.

We first study an auxiliary series called $S$-series defined by difference equation with $A(z)$. Initially introduced in \cite{Zhang} for any complex simple lie algebra in order to construct R-matrices for shifted Yangians, its coproduct  is already known for $ \ssl_2$. 

\begin{definition}
    \cite[Proposition 4.1]{Zhang} The series $S(z)$ is the unique power series in $z^{-1}$ of the constant term 1 with coefficients in $Y^{=}$ that solves the following difference equation :
    \[ S(z+1) = S(z)A(z)K(z) \in 1 + z^{-1}Y^{=}[[z^{-1}]] \]
    where $K(z) := \exp\po -\frac{1}{2}h_0 \sum\limits_{k>0} \frac{(-z)^{-k}}{k} \pf.$
\end{definition}

\begin{proposition}\label{coproduitS}
    \cite[Example 5.8]{Zhang} We have :
    \[ \Delta(S(z)) = (1 \otimes S(z))\exp(-z^{-1}x^{-}_{0}\otimes x^{+}_{0})(S(z) \otimes 1).\]
\end{proposition}

\begin{remark}
    Moreover, the series $K(z)$ is defined as the exponential of a primitive element ($\Delta(h_0) = h_0 \otimes 1 + 1 \otimes h_0$). It is then group-like :
\[ \Delta(K(z)) = K(z) \otimes K(z).\]
\end{remark}

Before computing the coproduct of $A(z)$, let us recall the following relations in $Y$. 

\begin{proposition}\label{permutation}
    \cite[(2.2),(4.10)]{Zhang}
    All the coefficients of $A(z)$ and $S(z)$ commute and
    \begin{align*}
        S(z)x^{-}_{0} &= (x^{-}_{0} - x^-_1 z^{-1}) S(z), \\
        x^{+}_{0}S(z) &= S(z)(x^{+}_{0}-x^+_1 z^{-1}), \\
        A(z)x^{-}_{0} &= x^{-}_{0}A(z) + \sum\limits_{k \geqslant 0}x^-_k z^{-k-1} A(z), \\ 
        x^{+}_{0}A(z) &= A(z)x^{+}_{0} + A(z)\sum\limits_{l \geqslant 0} x^+_l z^{-l-1}.
    \end{align*}
\end{proposition}

\subsection{Simplification of the problem}
To compute the coproduct of $A(z)$, we will need some binomial formulas. Let us return to a more general linear algebra problem and then apply it to our reasoning.

\begin{proposition}
    Let $K$ be a field and $A$ be a $K$-algebra. Let $a,b \in A$ and $c \in K$ be such that $ab - ba =ca^2 $.
    Then, for all $n\in \N$ : 
    \[ (a+b)^n = \sum\limits_{k=0}^n \binom{n}{k} \prod\limits_{j=0}^{k-1} (1 - jc) a^kb^{n-k}.\]
\end{proposition}

\begin{proof}
    Denote $(a+b)^n = \sum\limits_{k=0}^n \alpha_{k,n} a^kb^{n-k}$. To prove the proposition, we just have to check if the sequence $(\alpha_{k,n})_{k,n\in\N}$ satisfies the same recurrence relation as the sequence $(\beta_{k,n})_{k,n\in \N}$ defined by $\beta_{k,n} := \binom{n}{k}\prod\limits_{j = 0}^{k-1}(1-jc)$.

    Indeed, both satisfy the recurrence relation, which admits a unique solution :
    \begin{align*} u_{0,0} &= u_{1,0} = u_{0,1} = 1; \\ u_{k,n+1} &= u_{k,n} + (1-(k-1)c)u_{k-1,n}.
    \end{align*}
\end{proof}

\subsection{Coproduct}
We aim to compute $\Delta(A(z))$. The series $K(z)$ is group-like : we then start by computing the coproduct of $\Tilde{A}(z) := S(z)^{-1} S(z+1)$. 

\begin{proposition}\label{premform}
    \[ \Delta(\Tilde A(z)) = \po S(z)^{-1} \otimes S(z+1) \pf \po 1 + z^{-1}(z+1)^{-1}x^{-}_{0} \otimes x^{+}_{0} \pf \po S(z+1) \otimes S(z)^{-1} \pf.  \]
\end{proposition}
\begin{proof}
From proposition \ref{coproduitS} :
\begin{align*}  \Delta(\Tilde A(z)) = \sum\limits_{k,l} \po  S(z)^{-1} \frac{z^{-k}}{k!} \frac{\po -(z+1)  \pf^{-l} }{l!} (x^{-}_{0})^{k+l} S(z+1) \pf   \otimes \po (x^{+}_{0})^k S(z)^{-1} S(z+1) (x^{+}_{0})^l \pf. 
\end{align*}

Applying proposition \ref{permutation} many times, we get :
\begin{align*}
    (x^{+}_{0})^k S(z+1) &= S(z+1) \po x^{+}_{0} - x^+_1 (z+1)^{-1} \pf^k \\
    S(z)^{-1} (x^{+}_{0})^l &= \po x^{+}_{0} - x^+_1 z^{-1} \pf^l S(z)^{-1}.
\end{align*}

Then we can write
\begin{align*} \Delta(\Tilde{A}(z))  = &\po S(z)^{-1} \otimes S(z+1) \pf \po \sum\limits_{k = 0}^\infty \frac{z^{-k}}{k!} (x^{-}_{0})^k \otimes \po x^{+}_{0} - x^+_1 (z+1)^{-1}  \pf^k \pf  \\ &\po \sum\limits_{l=0}^\infty \frac{\po -(z+1) \pf^{-l}}{l!} (x^{-}_{0})^l \otimes \po x^{+}_{0} - x^+_1 z^{-1} \pf^l \pf \po S(z+1) \otimes S(z)^{-1} \pf.
\end{align*}

Let
\[ X :=  z^{-1} x^{-}_{0} \otimes x^{+}_{0}; \quad Y = -z^{-1}(z+1)^{-1} x^{-}_{0} \otimes x^+_1\]
in order to write the coproduct :

\[ \Delta(\Tilde A(z)) = \po S(z)^{-1} \otimes S(z+1) \pf \exp(X+Y)\exp(-z(z+1)^{-1} X-Y)\po S(z+1) \otimes S(z)^{-1} \pf. \]

Let $\lambda := (z+1)^{-1}$. Then
    \[ 1 + z^{-1}(z+1)^{-1}x^{-}_{0} \otimes x^{+}_{0} = 1 + \lambda X. \]
    and
    \[ XY - YX = \lambda X^2.\]

From the current computed coproduct formula, we just have to prove  
\[ \exp\po X + Y\pf = \po 1 + \lambda X \pf\exp\po(1-\lambda) X + Y\pf.\]

Recall that we work in $Y\otimes Y$, which is a $\Z$-bigraduated algebra. Therefore, to prove the previous equality, we prove it for each term according to their biweight. We then want to prove
\[ \tag{$\dagger$}\frac{1}{n!} \po X+Y \pf^n = \frac{1}{n!}((1-\lambda) X + Y)^n + \frac{1}{(n-1)!}\lambda X\ \po (1-\lambda) X + Y \pf^{n-1}\]
for all $n$ (it is the term of biweight $(n\alpha,-n\alpha)$).

But by the previous subsection :
\[ (X+Y)^n = \sum\limits_{k=0}^n \binom{n}{k} \prod\limits_{j=0}^{k-1} (1 - j \lambda) X^k Y^{n-k}. \]

We can then check the equality $(\dagger)$, by fixing $k \in \{0,\dots,n\}$ and looking at the coefficient of $X^k Y^{n-k}$.
\end{proof}

Now, let us prove one of the main results of this paper.
\begin{theoreme}\label{formulecopro}
    Let $m,n$ be two integers. Let $a(z)$ and $b(z)$ be complex Laurent series in $z^{-1}$ of leading terms $z^{2m}$ and $z^{2n}$ respectively. We have :
    \[ \Delta(A^{ab}(z)) = A^a(z) \otimes A^b(z) + \co A^a(z),x^{-}_{0} \cf \otimes \co x^{+}_{0}, A^b(z) \cf.\]
\end{theoreme}

\begin{proof}
    There exists a unique $a^*(z) \in z^m(1+z^{-1}\C[[z^{-1}]])$ such that 
    \[ a(z) = a^*(z)a^*(z-1).\]

    Moreover, the operation $a(z) \mapsto a^*(z)$ is multiplicative :
    \[ (ab)^*(z) = a^*(z)b^*(z).\]
    
    Finally, one has
    \[ A^a(z) = a^*(z)A^1(z).\]
    Therefore, we just have to prove the result for $A^1(z) = A(z)$.
    
    We just have to reorder the terms from proposition \ref{premform}. Let us recall the following relations :
    \[ S(z)^{-1}x^{-}_{0} = \sum\limits_{k\geqslant 0} x^{-}_k z^{-k}S(z)^{-1}; \quad
    x^{+}_{0}S(z)^{-1} = S(z)^{-1}\sum\limits_{k \geqslant 0} x^{+}_k z^{-k}.
    \]
    We can then write :
    \[ \Delta(A(z)) = \po 1 \otimes S(z)^{-1}S(z+1) \pf \po 1 + z^{-1}(z+1)^{-1} \sum\limits_{k,l \geqslant 0} z^{-k} x^{-}_k \otimes z^{-l} x^{+}_l \pf \po A(z)\otimes K(z) \pf. \]
    Let us invert the terms :
    \[ x^{+}_l K(z) = \po \frac{z}{z+1} \pf^{-1} K(z) x^{+}_l.\]
    Then :
    \[ \Delta(A(z)) = A(z) \otimes A(z) + \sum\limits_{k,l \geqslant 0} z^{-k-1} x^{-}_k A(z) \otimes A(z) x^{+}_l z^{-l-1}.\]
    Combining with proposition \ref{permutation} :
    \begin{align*} 
    A(z)x^{-}_{0} &= x^{-}_{0}A(z) + \sum\limits_{k \geqslant 0}x^{-}_k z^{-k-1} A(z) \\  
     x^{+}_{0}A(z) &= A(z)x^{+}_{0} + A(z)\sum\limits_{l \geqslant 0} x^{+}_l z^{-l-1}, \end{align*}
    we then get the desired formula.
\end{proof}

\section{Coproducts for truncated Yangians}
We arrive at the main result of this article for Yangians : the compatibility between truncation and coproducts.

\begin{theoreme}\label{compatibilityangian}
    Let $a(z) \in z^{2m}(1 + z^{-1}\C[[z^{-1}]])$ and $b(z) \in z^{2n}(1 + z^{-1}\C[[z^{-1}]])$. Then the coproduct of the Yangian induces an algebra morphism :
    \[ Y^{ab} \longrightarrow Y^a\otimes Y^b.\]
\end{theoreme}

\begin{proof}
    We already have :
\[\xymatrix{
    Y \ar[r]^\Delta \ar[d]^{\pi_{ab}} & Y \otimes Y \ar[d]^{\pi_a \otimes \pi_b} \\
    Y^{ab} & Y^a\otimes Y^b
  }\]

where $\pi_a$ (resp. $\pi_b$) is the canonical projection from Yangian $Y$ to its truncation $Y^a$ (resp. $Y^b)$.

We then just have to prove that $\pi_a\otimes \pi_b (\Delta(A^{ab}(z))$ is a polynomial in $z$ with coefficients in $Y^a \otimes Y^b$. 

But, by definition, $\pi_a(A^a(z))$ and $\pi_b(A^b(z))$ are polynomials in $z$, and so are $\pi_a \po \co A^a(z),x^{-}_{0} \cf \pf$ and $\pi_b \po \co x^{+}_{0},A^b(z) \cf \pf$. Then the theorem \ref{formulecopro} concludes.
\end{proof}

As an application of this result, we will consider the representations of $Y$ which come from the irreducible representations of $\ssl_2$ and compute the actions of the $A$-series. Let $\alpha, \beta \in \C$ such that $\alpha - \beta \in \N$. There is a representation $L(\alpha,\beta)$ with  basis $(v_0,v_1,\dots,v_{\alpha-\beta})$ such that \cite[Proposition 2.6]{ChariPressleyYangian}:
\[ h(z)\cdot v_i = \frac{(z+\alpha)(z+\beta-1)}{(z+\beta +i)(z+\beta+i-1)}v_i; \quad x_0^-\cdot v_i = v_{i+1}; \quad x_0^+ \cdot v_i = i(\alpha - \beta - i +1)v_{i-1}.\]
Set $a(z) = (z+\alpha)(z+\beta-1)$. Then the action of the GKLO series $A^a(z)$ is diagonal and polynomial :
$$A^a(z) v_i = \po z+\beta + i\pf v_i.$$

We now compute the action of $A(z)$ on the tensor products of such representations. 
    Let $\alpha_1,\alpha_2,\beta_1,\beta_2 \in \C$ such that $\alpha_1 - \beta_1 \in \N$ and $\alpha_2 - \beta_2 \in \N$. Let $(v_i)_{0\leqslant i \leqslant \alpha_1 - \beta_1}$ and $(w_j)_{0 \leqslant j \leqslant \alpha_2 - \beta_2}$ be the corresponding bases of $L(\alpha_1,\beta_1)$ and $L(\alpha_2,\beta_2)$. Set 
    \[ a_1(z) := (z+\alpha_1)(z+\beta_1-1) ; \quad a_2(z) := (z+\alpha_2)(z+\beta_2-1).\]
    Then $A^{a_1a_2}(z)$ acts on $L(\alpha_1,\beta_1)\otimes L(\alpha_2,\beta_2)$ as follows :
    \[ A^{a_1a_2}(z) \cdot v_i\otimes w_j = (z+\beta_1 +i)(z+\beta_2 + j)v_i\otimes w_j + j(\alpha_2 - \beta_2 - j +1)v_{i+1}\otimes w_{j-1}.
    \]

\section{Quantum affine algebra}
We now study the case of the quantum affine algebra. The ideas are similar, with some slight changes : a modified definition of truncation and the use of $q$-exponential.

Let $q \in \C^\times$ which is not a root of unity.

\begin{definition}
    The \textbf{quantum affine algebra} $U_q$ is the associative algebra defined by generators $x^\pm_{m}$, $\phi^\pm_m$ for $m \in \Z$ and relations for $\epsilon \in \{\pm 1\},n,m \in \Z$ :
    \begin{gather*}
        \phi^+_m = 0 \text{ if } m < 0; \quad \phi_m^- = 0 \text{ if } m > 0; \quad \phi_0^+\phi_0^- = 1 ; \\
        \co \phi^\pm_m, \phi^\epsilon_n\cf = 0; \quad \co x^+_m, x^-_n \cf = \frac{\phi^+_{m+n} - \phi^-_{m+n}}{q - q^{-1}}; \\
        \phi^\epsilon_{m+1}x_n^\pm - q^{\pm2} \phi^\epsilon_m x^\pm_{n+1} = q^{\pm2}x_n^\pm\phi^\epsilon_{m+1} -  x^\pm_{n+1} \phi_m^\epsilon; \\
        x^\pm_{m+1}x_n^\pm - q^{\pm2} x^\pm_m x^\pm_{n+1} = q^{\pm2}x_n^\pm x^\pm_{m+1} - x_m^\pm x^\pm_{n+1}.
    \end{gather*}
\end{definition}

Define the generating series :
\[ x^\pm(z) := \sum\limits_{n\in\Z} x^\pm_nz^{n}, \quad \phi^\pm(z) := \sum\limits_{n\in\Z} \phi^\pm_nz^{n}.\]

Define the Drinfeld-Cartan elements $h_n$ for $n\in \Z \setminus \{0\}$  by :
\[ \phi^\pm(z) = \phi^\pm_0\exp \po \pm(q-q^{-1})\sum\limits_{\pm s > 0} h_s z^s \pf.\]

The quantum affine algebra is equipped with a Hopf algebra structure whose coproduct is determined by :
\begin{align*}
    \Delta(x_0^\pm) &= x_0^\pm \otimes 1 + 1 \otimes x^\pm_0; \quad \Delta(\phi^\pm_0) = \phi^\pm_0 \otimes \phi^\pm_0; \\
    \Delta(h_{-1}) &= 1 \otimes h_{-1} + h_{-1} \otimes 1 + (q^2 -q^{-2})x_0^- \otimes x_{-1}^+ ; \\
    \Delta(h_1) &= 1 \otimes h_1 + h_1 \otimes 1 - (q^{2}-q^{-2})x_1^- \otimes x_0^+.
\end{align*}

\begin{definition}\cite[Proposition 5.5]{FrenkelHernandez},\cite[9.2]{Hernandez}
    \[ T^\pm(z) := \exp \po \pm(q-q^{-1}) \sum\limits_{\pm s > 0}  \frac{1}{q^{2s} - q^{-2s}} h_s z^s \pf.\]
\end{definition}

\begin{remark}
    The series $T^-(z)$ was introduced in \cite{FrenkelHernandez}, then $T^+(z)$ in \cite{Hernandez}. The convention used here is different and is introduced in \cite{ZhangHernandez}.
\end{remark}

The coproduct of the $T$-series has been studied in \cite{Zhang}. To write it properly, we need the definition of $q$-exponential. Define :
\[ (n)_q := \frac{q^{2n}-1}{q^2-1}; \quad (n)_q! := \prod\limits_{k= 1}^n (k)_q.\]
Let $x$ an element of an associative algebra $B$. For $z$ formal variable, define the $q$-exponential by :
    \[ \exp_q(zx) := \sum\limits_{n \geqslant 0} \frac{1}{(n)_q!}x^n z^n \in B[[z^{}]].\]

\begin{proposition}\cite[Example 9.6]{Zhang}\label{coproduitT}
    
\begin{align*}
    \Delta(T^-(z)) &= \po 1\otimes T^-(z) \pf\exp_q\po(q-q^{-1}) x^{-}_{0} \otimes x^{+}_{-1} z^{-1} \pf \po T^-(z) \otimes 1 \pf
\end{align*}

\end{proposition}

\begin{definition}\cite{Finkeliuk}\label{troncatureaffine}
    Let $n \in \N$ and $a(z)$ a complex polynomial with dominant term $a_{2n}z^{2n}$ and non-zero constant term $a_0$. Define the GKLO series $\A^{a,+}(z)$ and $\A^{a,-}(z)$ power series of constant term $1$ in $z$ and $z^{-1}$ respectively by :
    $$ \phi^+(z) = \frac{\phi_0^+a(z)a_0^{-1}}{\A^{a+}(a)\A^{a+}(zq^2)}; \quad
    \phi^-(z) = \frac{\phi_0^-a(z)a_{2n}^{-1}z^{-2n}}{\A^{a-}(z)\A^{a-}(zq^2)}.$$

    One can write
    \[ \A^{a\pm}(z) = \sum\limits_{k \in \Z}\A_k^{a,\pm} z^{k}\]
    so that $\A_k^{a,+} = 0 = \A_{-k}^{a,-}$ if $k < 0$ and $\A_0^{a,\pm} = 1$. 
    The truncated quantum affine algebra $U_q^a$ is defined as the quotient algebra of $U_q$ by the following relations for  $k \in \Z$ :
    \begin{gather}
        \A^{a+}_k = 0 \text{ for } k > n; \quad \A^{a-}_k = \A^{a-}_{-n}\A^{a+}_{k+n} \text{ for } k \in \Z; \\
        a_{2n}(\phi_0^+)^2 = a_0(\A^{a,+}_{n})^2q^{2n}.    
    \end{gather}
\end{definition}

\begin{remark}
    This definition follows the convention in \cite[(3.20-3.23)]{ZhangHernandez}, which is a simplification of the definition from \cite[Definition 8.12]{Finkeliuk}. The inverse of the series $\A^{1,-}(z)$ first appeared in \cite[(3.11)]{Frenkikhin}.

\end{remark}

As for the Yangian, getting a formula for $A^\pm(z) := \A^{1,\pm}(z)$ is enough to deduce one for $\A^{a\pm}(z)$. We shall need the following relations.
\begin{proposition}\cite[Proposition 3.2]{ZhangHernandez}\label{interversionquantique}   
    We have :
    $$ A^+(z) = \frac{T^+(zq^{-2})}{T^+(z)}; \quad 
    A^-(z) = \frac{T^-(zq^{-2})}{T^-(z)}.$$
    Moreover for all $n \in \Z$ :
    \begin{gather}
        T^-(z)^{-1} x_n^- = \sum\limits_{k\geqslant 0} x_{n-k}^-z^{-k}T^{-}(z)^{-1}; \quad x_{n}^+T^-(z)^{-1} = T^-(z)^{-1} \sum\limits_{l \geqslant 0} x_{n-l}^+z^{-l}; \\
        \co A^-(z),x^{-}_{n} \cf = (1-q^2)\sum\limits_{k \geqslant 1} x^-_{n-k}z^{-k}A^-(z) \\
        \co x_n^+,A^-(z) \cf = (1-q^2)A^-(z)\sum\limits_{k \geqslant 1} x_{n-k}^+z^{-k}.
    \end{gather}
\end{proposition}

\begin{corollaire}\label{twoformulas}
    We have :
    \[ \co A^-(z),x_n^- \cf = z^{-1}\co A^-(z),x_{n-1}^- \cf_{q^2}; \quad \co x_n^+, A^-(z)  \cf = z^{-1} \co x_{n-1}^+,A^-(z) \cf_{q^2} \]
    where $\co x,y \cf_{q^2} = xy - q^2yx$ for all $x,y \in U_q$.
\end{corollaire}

\begin{proof}
    Let us prove the result for $x^-_n$. According to (3) in proposition \ref{interversionquantique} :
    \begin{align*} \co A^-(z), x_n^- \cf_{q^2} &= (1-q^2)\sum\limits_{k \geqslant 0} x_{n-k}^-z^{-k} A^-(z) \\
    &= z(1-q^2)\sum\limits_{k\geqslant 1} x_{n+1-k} z^{-k} A^-(z) \\
    &= z\co A^-(z), x_{n+1}^- \cf
    \end{align*}
    last equality coming from (4) of the same proposition.
\end{proof}

\section{Formulas of coproducts for $U_q$}
As we have seen in proposition \ref{coproduitT}, the coproducts of $T$-series depend on $q$-exponential. Then let us recall two properties that will be useful.

\begin{proposition}\cite[(IV.2.4,IV.2.6)]{Kassel}\label{qexpo}\.
    Let $B$ be an associative algebra. Let $a,b,c \in B$ such that $ab = q^2ba$. Then 
    \begin{gather} \exp_q(z(a+b)) = \exp_q(zb)\exp_q(za); \quad \exp_q(zc) = \exp_{q^{-1}}(-zc); \\
    \frac{\exp_q(zc)}{\exp_q(q^{-2}zc)} = 1 + (1-q^{-2})zc
    \end{gather}
\end{proposition}

Let us start by computing $\Delta(A^-(z))$. We will then use this result to determine the coprodut of $A^+(z)$. 

To do so, we could do exactly the same calculations as we did for Yangian. However, in the case of the quantum affine algebra, we can use the properties of $q$-exponential to simplify computations.

\begin{proposition}
    \[ \Delta\po A^-(z) \pf = \po T^-(z)^{-1} \otimes T^-(zq^{-2}) \pf \po 1 + (1-q^2)^2q^{-1}z^{-1}x_0^-\otimes x_{-1}^+ \pf \po T^-(zq^{-2}) \otimes T^-(z)^{-1} \pf. \]    
\end{proposition}

\begin{proof}
Combining proposition \ref{interversionquantique}
 and proposition \ref{coproduitT}, we have : 
 \begin{align*} \Delta(A^-(z)) &=  \sum\limits_{k,l \geqslant 0}  T^-(z)^{-1}\frac{\po q^{-1}-q\pf^k}{\po k\pf_{q^{-1}}!} z^{-k} (x^{-}_{0})^k \frac{\po q - q^{-1} \pf^l}{\po l\pf_{q}!} z^{-l} q^{2l} (x^{-}_{0})^l T^-(zq^{-2})  \\ 
  & \quad \otimes (x^{+}_{-1})^k T^-(z)^{-1} T^-(zq^{-2})(x^{+}_{-1})^l \\
&= \po T^-(z)^{-1} \otimes T^-(zq^{-2}) \pf \exp_{q^{-1}}(X + Y)\exp_{q}(-q^2X - Y) \po T^-(zq^{-2})\otimes T^-(z)^{-1} \pf
\end{align*}
where $X := (q^{-1} - q)z^{-1}x^{-}_{0}\otimes x^{+}_{-1}$ and $Y := (q-q^{-1})z^{-2}q^2x^{-}_{0}\otimes x^+_{-2}$. Since $x^+_{-1} x_{-2}^+ = q^2x_{-2}^+x_{-1}^+$, we have $XY = q^2YX$. We simplify the right-hand side using proposition \ref{qexpo} :
    \begin{align*} \exp_{q^{-1}}(X+Y)\exp_q(-q^2X-Y) &= \exp_{q^{-1}}(X)\exp_{q^{-1}}(Y)\exp_{q}(-Y)\exp_q(-q^{2}X) \\
    &= \frac{\exp_{q^{-1}}(X)}{\exp_{q^{-1}}(q^2X)} = 1 + (1-q^2)X.
    \end{align*}
\end{proof}

\begin{theoreme}\label{coproduitaffinemoins}
    Let $a(z)$ and $b(z)$ complex polynomials of non-zero constant terms. Then :
    \begin{gather*}  \Delta(\A^{ab-}(z)) = \A^{a-}(z)\otimes \A^{b-}(z) + q^{-1}z \co \A^{a-}(z),x_{1}^- \cf \otimes \co x^+_{0},\A^{b-}(z) \cf; \\
    \Delta(\A^{ab+}(z)) = \A^{a+}(z)\otimes \A^{b+}(z) + qz^{-1} \co \A^{a+}(z),x_{0}^- \cf \otimes \co x^+_{-1},\A^{b+}(z) \cf. 
    \end{gather*}
\end{theoreme}
    
\begin{proof}
    We have $\A^{a+}(z) = a^*(z)A^+(z)$ where $a^*(z)$ is the unique power series in $z$ of constant term $1$ that solves the difference equation $$a(z) = a^*(z)a^*(zq^{-2}).$$ The map $a(z) \mapsto a^*(z)$ is multiplicative. Similarly, $\A^{a-}(z) = a^\sharp(z)A(z)$ where $a^\sharp(z)$ is the unique power series in $z^{-1}$ of constant term $1$ that solves the difference equation $$a(z) = a^\sharp(z)a^\sharp(zq^{-2}).$$ The map $a(z) \mapsto a^\sharp(z)$ is multiplicative.
    As in the case of Yangian, it suffices to prove the formula for $A^\pm(z)$.

    We use (3) of proposition \ref{interversionquantique} to reorder the terms of the tensor product. We then get an expression of the coproduct depending only on $A^-(z)$ :
\[ \Delta \po A^-(z) \pf = A^-(z) \otimes A^-(z) + (1-q^2)^2 q^{-1} \sum\limits_{k,l \geqslant 0} x^-_{-k} z^{-k} A^-(z) \otimes A^-(z)x^+_{-l-1} z^{-l-1} . \]
    Then (4) of proposition \ref{interversionquantique} concludes.

    To compute $\Delta(A^+(z))$, we use the anti-isomorphism $\Omega$ introduced in \cite[(1.3)]{Beck} as the algebra and coalgebra anti-isomorphism $\Omega: U_q \longrightarrow U_{q^{-1}}$ determined by
    $$ x_n^+ \mapsto x_{-n}^-, \quad x_n^- \mapsto x_{-n}^+, \quad h_n \mapsto h_{-n}. $$ Moreover $\Omega(A^-(z^{-1})) = A^+(z)$, so the coproduct of $A^+(z)$ follows from the coproduct of $A^-(z)$.
\end{proof}

\section{Coproducts for truncated quantum affine algebras}
As for Yangians, we shall now prove the compatibility between truncation and coproduct.
\begin{theoreme}\label{compatibilityaffine}
    Let $a(z)$ and $b(z)$ complex polynomials of non-zero constant terms of even degrees. Then the coproduct of the quantum affine algebra induces an algebra morphism :
    \[ U_q^{ab} \longrightarrow U_q^a \otimes U_q^b.\]
\end{theoreme}

\begin{proof}
    We will check if the three conditions of definition \ref{troncatureaffine} are satisfied. We write $2m_a$ and $2m_b$ as degrees of $a(z)$ and $b(z)$ respectively. 

    For the first one, this is the same argument as for Yangian. Just note that, for $\A^{a-}(z)$ of degree $m_a$
    $\co \A^{a-}(z),x_1^- \cf$ is of degree $m_a - 1$ (and similarly for $b(z)$). 

    For the second one, note that 
    \begin{gather*} 
    \Delta(\A^{ab-}_{-m_a - m_b}) = \A^{a-}_{-m_a} \otimes\A^{b-}_{-m_b}; \\
    \quad \Delta(\A^{ab-}_{k})= \sum\limits_{i \in \Z} \A^{a-}_i \otimes \A^{b-}_{k-i} + q^{-1}\sum\limits_{i\in \Z} \co \A^{a-}_{i-1},x_{1}^- \cf \otimes \co x_0^+,\A^{b-}_{k-i} \cf; \\
    \quad \Delta(\A^{ab+}_{m_a+m_b+k})= \sum\limits_{i \in \Z} \A^{a+}_{m_a +i} \otimes \A^{b+}_{m_b + k-i} + q^{-1}\sum\limits_{i\in \Z} \co \A^{a-}_{m_a + i},x_{0}^- \cf \otimes \co x_{-1}^+,\A^{b-}_{m_b + k-i + 1} \cf.
    \end{gather*}

    We start by computing, using the second relation of the definition \ref{troncatureaffine} in $U_q^a$ and $U_q^b$, the coproduct $\Delta(\A^{ab-}_{-m_a-m_b}\A^+_{m_a + m_b + k})$ which is equal to :
    \[ \sum\limits_{i} \A_i^{a-} \otimes \A^{b-}_{k-i} + q \sum\limits_{j} \po \A_j^{a-} x_0^- - \A_{-m_a}^{a-}x_0^- \A_{m_a+j}^{a-} \pf\otimes \po \A_{-m_b}^{b-}x_{-1}^+ \A^{b-}_{m_b+k-j+1} - \A_{k-j+1}^{b+}x_{-1}^+\pf .\]

    Let us recall the relations of q-commutativity from proposition \ref{interversionquantique} :
    \[ \A_{-m_a}^{a-} x_0^- = q^2x_0^-\A^{a-}_{-m_a}; \quad \A_{-m_b}^{b-}x_{-1}^+ = q^{-2}x_{-1}^+\A^{b-}_{-m_b}.\]

    One can then check, using corollary \ref{twoformulas}
    \[ \A^{a-}_{j}x_0^- - q^2 x_0^- \A_j^{a-} = \co A_{j-1}^{a-},x_1^- \cf; \quad q^{-2} x_{-1}^+ A^{b-}_{k-j+1} - A^{b-}_{k-j+1} x_{-1}^+ = q^{-2} \co x_0^+,A^{b-}_{k-j} \cf \]
    from which we deduce the result.
    
    For the third one, we aim to check if
    \[ (ab)_{2(m_a+m_b)}(\pi_a\otimes \pi_b)\circ\Delta(\phi_0^+)^2 = (ab)_0 (\pi_a\otimes \pi_b)\circ  \Delta(\A^{ab+}_{m_a+m_b})^2q^{2(m_a+m_b)}.\]
    But, on one hand :
    \begin{align*} 
        (ab)_{2(m_a+m_b)}(\pi_a\otimes \pi_b)\circ\Delta(\phi_0^+)^2 &= a_{2m_a}(\phi_0^+)^2 \otimes b_{2m_b} (\phi_0^+)^2 \\
        &= a_0(\A_{m_a}^{a+})^2 q^{2m_a} \otimes b_0 (\A_{m_b}^{b+})^2 q^{2m_b}
    \end{align*}
    and, on the other hand :
    \begin{align*}
        (ab)_0(\pi_a\otimes \pi_b)\circ\Delta(\A_{m_a+m_b}^{ab+})^2q^{2(m_a+m_b)} = a_0(\A_{m_a}^{a+})^2 q^{2m_a} \otimes b_0 (\A_{m_b}^{b+})^2q^{2m_b}.
    \end{align*}
\end{proof}

As an application of this result, we will consider the representations of $U_q$ induced from irreducible representations of $\ssl_2$. For $a \in \C^\times$ and $n \in \N$, let us describe such a representation of dimension $n+1$ by a basis (as in \cite[Example 5.6]{ZhangHernandez}) with the action :
\[ \phi^{\pm}(z) \cdot v_i = q^{n-2i} \frac{(1-q^{-n-1}az)(1-q^{n+1}az)}{(1-q^{n-1-2i}az)(1-q^{n+1-2i}az)}v_i; \quad (x_0^-)^i \cdot v_0 = v_i.\]
We denote by $L_n(a)$ this representation. It is the finite dimensional irreducible representation associated to the Drinfeld polynomial $(1-q^{n-1}az)(1-q^{n-3}az)\dots(1-q^{-n+1}az)$ in \cite{ChariPressley}.

Set $f(z) := (1-q^{-n-1}az)(1-q^{n+1}az)$. Then the action of $\A^{f\pm}(z)$ is diagonal and polynomial :
\[ \A^{f+}(z)v_i = (1-q^{n-1-2i}az)v_i; \quad \A^{f-}(z)v_i = (1-q^{-n+1+2i}a^{-1}z^{-1})v_i. \]
As a consequence, $L_n(a)$ is representation over $U_q^f$.

Let $b \in \C^*$ and $m \in \N$. Consider the representation $L_m(b)$ with basis $(w_j)_{0 \leqslant j \leqslant m}$ as above. We will need a new basis defined by $$v'_j : =\po(x^+_{-1})^j \cdot w_m\pf_{0 \leqslant j \leqslant m}.$$ Set $g(z) := (1-q^{-m-1}bz)(1-q^{m+1}bz)$. The action of $\A^{fg+}(z)$ on $L_n(a)\otimes L_m(b)$ is given by :
\[ \A^{fg+}(z)\cdot v_i\otimes v'_j = (1-q^{n-1-2i}az)(1-q^{-m-1+2j}bz) v_i \otimes v'_j + zabq^{n-m-2(i-j)}(q-q^{-1})^2 v_{i+1}\otimes v'_{j+1}. \]
By theorem \ref{compatibilityaffine}, the tensor product $L_n(a) \otimes L_m(b)$ is a module over $U_q^{fg}$, we can then apply (1) of definition \ref{troncatureaffine} to compute the action of $A^{fg-}(z)$ :
\[ \A^{fg-}(z)\cdot v_i\otimes v'_j = (1-q^{-n+1+2i}a^{-1}z^{-1})(1-q^{m+1-2j}b^{-1}z^{-1}) v_i \otimes v'_j + z^{-1}(1-q^2)^2 v_{i+1}\otimes v'_{j+1}. \]

\section{Compatibility in the shifted cases}
Our main results can be adapted to the shifted case. We state the results without proofs, as they follow from the non-shifted cases by standard zigzag arguments.

\subsection{Shifted Yangians}
\begin{definition}
    Let $\mu \in \Z$. The \textbf{shifted Yangian} $Y_\mu$ associated to $\ssl_2$ is the associative algebra defined by generators $x^-_n,x^+_n,\xi_p$ for all $n \in \N$, $p \in \Z$ satisfying the following relations :
    \begin{align*}
        \co \xi_p,\xi_q\cf = 0, &\quad \co x^+_n,x^-_m \cf = \xi_{n+m}, \\ 
        \co \xi_{p+1},x^\pm_n \cf - \co \xi_p,x^\pm_{n+1} \cf &= \pm \po \xi_p x^\pm_n + x^\pm_n\xi_p \pf,  \\
        \co x^\pm_{m+1},x^\pm_n \cf - \co x^\pm_m,x^\pm_{n+1} \cf &= \pm \po x^\pm_m x^\pm_n + x^\pm_n x^\pm_m \pf, \\
        \xi_{-\mu-1} = 1, &\quad  \xi_p = 0 \text{ for }p < - \mu - 1.
    \end{align*}
\end{definition}

Define the generating series :
\[ \xi(z) := \sum\limits_{p \in \Z} \xi_p z^{-p-1}.\]
This is a Laurent series in $z^{-1}$ of leading term $z^\mu$. Its coefficients generate a commutative subalgebra of $Y_\mu$, denoted by $Y^=_\mu$.

The standard Yangian correspond to the case $\mu = 0$. A family of algebra homomorphisms $$\Delta_{\mu,\nu} : Y_{\mu + \nu} \longrightarrow Y_\mu \otimes Y_\nu$$ for shifted Yangians has been introduced in a compatible way with the standard Yangian \cite[Corollary 3.16, Theorem 4.12, Proposition 4.14]{shiftedcoproduct}.

\begin{definition}\cite[Lemma 2.1]{GKLO}
    Let $m \in \N, \mu \in \Z$ and $a(z)$ a complex Laurant series in $z^{-1}$ of leading term  $z^{2m+ \mu}$. The GKLO series associated to $a(z)$ is the Laurent series in $z^{-1}$ of leading term $z^m$ and coefficients in $Y_\mu^=$ uniquely determined by :
    \[ \xi(z) = \frac{a(z)}{A_\mu^a(z)A_\mu^a(z-1)}.\]
     The truncated Yangian $Y_\mu^a$ is the quotient of $Y_\mu$ by the coefficients of strictly negative powers of $A^a_{\mu}(z)$.
\end{definition}

To compute the coproduct of $A$-series, we use the result of standard Yangian to get formulas for shifted Yangians by the zigzag properties of $\Delta_{\mu,\nu}$. We naturally get :

\begin{theoreme}
    For $\mu,\nu \in \Z, m,n \in \N$ and $a(z) \in z^{2m+\mu}(1 + z^{-1}\C[[z^{-1}]])$, $b(z) \in z^{2n + \nu}(1 + z^{-1}\C[[z^{-1}]])$, we have :
    \[ \Delta_{\mu,\nu}(A^{ab}_{\mu+\nu}(z)) = A^a_\mu(z) \otimes A^b_\nu(z) + \co A^a_\mu(z),x^{-}_{0}\cf \otimes \co x^{+}_{0},A^b_\nu(z) \cf. \]

    Moreover, $\Delta_{\mu,\nu}$ induces an algebra morphism $ Y^{ab}_{\mu + \nu} \longrightarrow Y^a_\mu \otimes Y^b_\nu$.
\end{theoreme}

\subsection{Shifted quantum affine algebras}
\begin{definition}
For $\mu \in \Z$, the \textbf{shifted quantum affine algebra}  $U_\mu$ associated to $\ssl_2$ is the associative algebra defined by generators $x^\pm_{m}$, $\phi^\pm_m$ for $m \in \Z$ satisfying the following relations, for $\epsilon \in \{\pm 1\},n,m \in \Z$ :
    \begin{align*}
        \phi^+_m = 0 \text{ if } m < 0; \quad \phi_m^- = 0& \text{ if } m > \mu; \quad \phi_0^+\phi_\mu^- \text{ is central and invertible}; \\
        \co \phi^\pm_m, \phi^\epsilon_n\cf = 0;& \quad \co x^+_m, x^-_n \cf = \frac{\phi^+_{m+n} - \phi^-_{m+n}}{q - q^{-1}}; \\
        \phi^\epsilon_{m+1}x_n^\pm - q^{\pm2} \phi^\epsilon_m &x^\pm_{n+1} = q^{\pm2}x_n^\pm\phi^\epsilon_{m+1} -  x^\pm_{n+1} \phi_m^\epsilon; \\
        x^\pm_{m+1}x_n^\pm - q^{\pm2} x^\pm_m &x^\pm_{n+1} = q^{\pm2}x_n^\pm x^\pm_{m+1} - x_m^\pm x^\pm_{n+1}.
    \end{align*}
\end{definition}

\begin{remark}\label{nonstandard}
    Contrary to Yangian, the case $\mu = 0$ is not the standard quantum affine algebra : condition $\phi^+_0\phi^-_0 = 1$ has been relaxed.
\end{remark}

Define the generating series :
\[ \phi^\pm(z) := \sum\limits_{n\in\Z} \phi^\pm_nz^{n}.\]

As for shifted Yangians, we have a family of algebra homomorphisms
$$ \Delta_{\mu,\nu} : U_{\mu + \nu} \longrightarrow U_\mu \otimes U_\nu$$ interpreted as a generalization of the standard coproduct is introduced in \cite[Theorem 10.26]{Finkeliuk}.

\begin{definition}\cite{Finkeliuk}
    Let $n\in \N, \mu \in \Z$ and $a(z)$ a complex polynomial with dominant term $a_{2n+\mu}z^{2n +\mu}$ and non-zero constant term $a_0$. Define the GKLO series $\A_\mu^{a,+}(z)$ and $\A_\mu^{a,-}(z)$ to be power series of constant term $1$ in $z$ and $z^{-1}$ respectively by 
    \[ \phi^+(z) = \frac{\phi_0^+a(z)a_0^{-1}}{\A_\mu^{a+}(z)\A_\mu^{a+}(zq^2)}; \quad
    \phi^-(z) = \frac{\phi_\mu^-a(z)a_{2n+\mu}^{-1}z^{-2n}}{\A_\mu^{a-}(z)\A_\mu^{a-}(zq^2)}. \]

    One can write
    \[ \A_\mu^{a\pm}(z) = \sum\limits_{k \in \Z}\A_{k}^{a,\pm} z^{k}\]
    so that $\A_k^{a,+} = 0 = \A_{-k}^{a,-}$ if $k < 0$ and $\A_0^{a,\pm} = 1$. 
    The truncated algebra $U_\mu^a$ is the quotient algebra of $U_\mu$ by the following relations for $k \in \Z$ :
    \begin{gather*}
        \A^{a+}_k = 0 \text{ for } k > n; \quad \A^{a-}_k = \A^{a-}_{-n}
        \A^{a+}_{k+n}; \\
        a_{2n + \mu}\phi_0^+ = a_0\phi_\mu^- (\A^{a,+}_{n})^2q^{2n}; \\
        \phi_0^+\phi_0^- = a_0^{-1}a_{2n+\mu}q^{-2n}.
    \end{gather*}
\end{definition}

\begin{remark}
    The definition of truncation used here follows \cite[(3.22)-(3.24)]{ZhangHernandez} which is a slight modification of \cite{Finkeliuk}. We refer to \cite[Remark 3.11]{ZhangHernandez} for a comparison with original definition of \cite{Finkeliuk}.
\end{remark}

\begin{theoreme}
    Let $\mu,\nu \in \Z$, $a(z)$ and $b(z)$ complex polynomials of non-zero constant terms such that $\deg(a(z)) - \mu$ and $\deg(b(z)) - \nu$ are even. We have :
    \begin{align*}
     \Delta_{\mu,\nu}(\A_{\mu+\nu}^{ab-}(z)) &= \A_\mu^{a-}(z)\otimes \A_\nu^{b-}(z) + q^{-1}z \co \A_\mu^{a-}(z),x_{1}^- \cf \otimes \co x^+_{0},\A_\nu^{b-}(z) \cf.\\
     \Delta_{\mu,\nu}(\A_{\mu+\nu}^{ab+}(z)) &= \A_\mu^{a+}(z)\otimes \A_\nu^{b+}(z) + qz^{-1} \co \A_\mu^{a+}(z),x_{0}^- \cf \otimes \co x^+_{-1},\A_\nu^{b+}(z) \cf.
    \end{align*}
    Moreover, $\Delta_{\mu,\nu}$ induces an algebra morphism :
        \[ U^{ab}_{\mu + \nu} \longrightarrow U^a_\mu \otimes U^b_\nu.\]
\end{theoreme}

\bibliography{biblio.bib}
\bibliographystyle{plain}
\end{document}